\documentclass[12pt,a4paper,reqno]{amsart}
\usepackage{preemble}

\date{\today}

\author{Stephen Cantrell}
\address{Department of Mathematics, 
University of Warwick,
Coventry, CV4 7AL, UK}
\email{stephen.cantrell@warwick.ac.uk}

\date{\today. 
}

\title[Sparse spectrally rigid sets for negatively curved manifolds]{Sparse spectrally rigid sets for negatively curved manifolds}

\begin{document}
\maketitle

\begin{abstract}
Suppose that $(M,\mathfrak{g})$ is a compact Riemannian manifold with strictly negative sectional curvatures. A subset of conjugacy classes $E \subset \conj(\pi_1(M))$ is called spectrally rigid if when two negatively curved Riemannian metrics $\mathfrak{g}_1, \mathfrak{g}_2$ on $M$ have the same marked length spectrum on $E$, then their marked length spectra coincide everywhere. In this work we show that there are arbitrarily sparse spectrally rigid sets and that they exist, in some sense, in every direction in $\pi_1(M)$.
\end{abstract}

\section{Introduction}
%------------------------------------------------------------------------------------------------------------------------------------------------------------------------------------------------------------------
Let $(M,\mathfrak{g})$ be a compact Riemannian manifold (possibly with totally geodesic boundary) with strictly negative sectional curvatures.
There is a bijection between the closed geodesics on $(M, \mathfrak{g})$ and the conjugacy classes $\conj(\G)$ in the fundamental group $\G = \pi_1(M)$. Through this bijection we can associate a length to each conjugacy class. That is, there is a function $\ell_\mathfrak{g} : \conj(\G) \to \R_{> 0}$ that sends each conjugacy class $[x] \in \conj(\G)$ to the length of the corresponding closed geodesic $\ell_\mathfrak{g}[x]$. The function $\ell_\mathfrak{g}$ is referred to as the marked length spectrum of $(M,\mathfrak{g})$. The well known marked length spectrum rigidity conjecture predicts that the isometry class of $\mathfrak{g}$ is determined by its marked length spectrum. That is, if two negatively curved Riemannian metrics $\mathfrak{g}_1, \mathfrak{g}_2$ on $M$ have the same marked length spectrum ($\ell_{\mathfrak{g}_1} [x]= \ell_{\mathfrak{g}_2}[x]$ for all $[x] \in \conj(\G)$) then $(M,\mathfrak{g}_1)$ and $(M,\mathfrak{g}_2)$ are isometric. This conjecture was resolved in dimension $2$ by Otal \cite{otal} and independently by Croke \cite{croke}. In higher dimensions, the conjecture was verified by Hamenst\"adt  when one of the metrics is locally symmetric \cite{hamenstadt}. Recently there has been much interest in the conjecture and related problems with progress made in various directions \cite{croke.dai}, \cite{rafi.duch},  \cite{bridgeman.canary}, \cite{g.l}, \cite{g.k.f}, \cite{yanlong}, \cite{gogolev.hertz}, \cite{cantrell.reyes.lsr},  \cite{Butt}.

It is natural to try to identify subsets of conjugacy classes which determine the marked length spectrum of a negatively curved metric $\mathfrak{g}$. We say that a collection of conjugacy classes $E \subset \conj(\G)$ is spectrally rigid if when $\ell_{\mathfrak{g}_1}[x] = \ell_{\mathfrak{g}_2}[x]$ for $[x] \in E$ then we must have that $\ell_{\mathfrak{g}_1}[x] = \ell_{\mathfrak{g}_2}[x]$ for all  $[x] \in \conj(\G)$. Examples of spectrally rigid sets include:
\begin{enumerate}
\item conjugacy classes with a representative in a normal subgroup $H < \pi_1(M)$ \cite{bonahonbook} or more generally a subgroup that has full limit set  in the universal cover \cite{yanlong};
\item conjugacy classes corresponding to a fixed homology class \cite{gogolev.hertz}; and,
\item subsets of conjugacy classes that have complement with sub-exponential growth \cite{noelle}, have positive upper density \cite{dilsavor.reber} or more generally have full exponential growth rate for all Riemannian metrics on $M$ \cite{cantrell.reyes.lsr}.
\end{enumerate} 
Note also that if we restrict to hyperbolic metrics on compact surfaces there are finite spectrally rigid sets coming from a pants decomposition, i.e. its Fenchel Nielsen coordinates. This is not the case in variable negative curvature (as we can perturb negatively curved metrics away from the union of any finite number of closed geodesics). 

%It would be interesting to ascertain the largest growth rate possible for non-spectrally rigid sets for negatively curved manifolds, i.e. what is the largest possible growth rate of $\#\{ \ell_{\mathfrak{g}}[x] < T : [x] \in E\}$ for $E$ a non-spectrally rigid set?
%We make the following observation. Suppose that $(V,\mathfrak{g})$ is a negatively curved, closed Riemannian surface. Then $\#\{[x] \in E :\ell_{\mathfrak{g}}[x] < T \}$ has exponential growth rate $h_\mathfrak{g} >0$ where $h_\mathfrak{g}$ is the entropy of the geodesic flow on $(V,\mathfrak{g})$. For any $\alpha < h_\mathfrak{g}$ we can find a non-spectrally rigid set $E$ with exponential growth rate $v_E = \limsup_{T\to\infty} \frac{1}{T} \log \#\{ [x] \in E : \ell_{\mathfrak{g}}[x] < T  \}$ such that $\alpha \le   v_E <  h_\mathfrak{g}$. To do so we fix a ball of radius $\epsilon > 0$ in $(S,\mathfrak{g})$ and let $E$ be the set of conjugacy classes corresponding to closed geodesics that do not enter this ball. The set $E$ is not spectrally rigid (as we can perturb the metric $\mathfrak{g}$ in this ball to obtain a new metric that is not isometric to $\mathfrak{g}$) and as $\epsilon \to 0$ the exponential growth rate of $E$ approaches $h_\mathfrak{g}$. It is likely a difficult question to ascertain whether there is a non-spectrally rigid set $E$ such that $\#\{ \ell_{\mathfrak{g}}[x] < T : [x] \in E\}$  has exponential growth rate $h_\mathfrak{g}$.

If $\mathfrak{g}$ is a negatively curved Riemannian metric on $M$ and $\alpha \subset \conj(\pi_1(M))$ is the collection of conjugacy classes lying in a fixed homology class then $\#\{ [x] \in \alpha : \ell_\mathfrak{g}[x] <T\}$ grows exponentially as $T\to\infty$ and in fact has the same exponential growth rate as  $\#\{[x] \in \conj(\pi_1(M)) : \ell_\mathfrak{g}[x] <T \}$ \cite{adachi.sunada}. The spectrally rigid sets listed in $(1), (2)$ and $(3)$ above are all large in this sense: they have strictly positive exponential growth rate, i.e. for each set $E$ mentioned in $(1), (2)$ and $(3)$  above
\begin{equation}\label{eq.gr}
\#\{[x] \in E : \ell_\mathfrak{g}[x] < T\} 
\end{equation}
grows exponentially quickly as $T\to\infty$. They are also large in the sense that they have full limit sets in the universal covers of the manifolds. This may lead us to suspect that spectrally rigid sets are necessarily large and it is therefore natural to search for small  spectrally rigid sets. This problem is given more motivation by the fact that there exist large non-spectrally rigid sets. Indeed, suppose that $(V,\mathfrak{g})$ is a negatively curved, closed Riemannian surface. Then the number of closed geodesics on $(V,\mathfrak{g})$ of length less than $T$ has exponential growth rate $h_\mathfrak{g} >0$ where $h_\mathfrak{g}$ is the topological entropy of the geodesic flow on $(V,\mathfrak{g})$. For any $\alpha < h_\mathfrak{g}$ there is a non-spectrally rigid set $E$ with exponential growth rate 
\[
v_E = \limsup_{T\to\infty} \frac{1}{T} \log \#\{ [x] \in E : \ell_{\mathfrak{g}}[x] < T  \}
\]
 such that $\alpha \le   v_E <  h_\mathfrak{g}$. To construct such a set, fix a ball of radius $\epsilon > 0$ in $(S,\mathfrak{g})$ and let $E$ be the set of conjugacy classes corresponding to closed geodesics that do not enter this ball. The set $E$ is not spectrally rigid (as we can perturb the metric $\mathfrak{g}$ in this ball to obtain a new metric that is not isometric to $\mathfrak{g}$) and as $\epsilon$ tends to $0$ the exponential growth rate of $E$ approaches $h_\mathfrak{g}$. It is likely a difficult question to ascertain whether there is a non-spectrally rigid set $E$ such that $\#\{ [x] \in E : \ell_{\mathfrak{g}}[x] < T \}$  has exponential growth rate $h_\mathfrak{g}$. We do not address this question in the current work but it would be interesting to investigate this problem further.

The author is aware of only two results for spectrally rigid sets with sub-exponential growth rates. However, these results are in different settings to that of negatively curved manifolds considered above. The first is a result of Bridgeman and Canary \cite{bridgeman.canary} that states that for Kleinien surfaces groups, the collection of simple closed geodesics form a spectrally rigid set. Note that simple closed geodesics on negatively curved surfaces do not form a spectrally rigid set as their union is not dense. The second is a rigidity result of Kapovich for the  Culler-Vogtmann Outer Space of free groups $F_N$ with rank $N \ge 2$ \cite{kap}. Kapovich showed that typical geodesic rays (in some sense) in $F_N$ form spectrally rigid sets for Outer Space. There are many other works that study rigidity sets for Outer Space \cite{smillie.vogtmann}, \cite{osspec1}, \cite{osspec3}, \cite{osspec2}. In both the results of Bridgeman-Canary and Kapovich the spectrally rigid sets have polynomial growth (i.e. the analogues of $(\ref{eq.gr})$ grow polynomially in $T$). 

In this work we generalise and improve Kapovich's result \cite{kap}. In the setting of negatively curved manifolds, we show that, surprisingly, arbitrarily small spectrally rigid sets exist in all directions in $\pi_1(M)$. 

Before we state this result we recall the following facts.
Let $(M,\mathfrak{g})$ be a compact Riemannian manifold with strictly negative sectional curvatures. The fundamental group $\pi_1(M)= :\G$ of $M$ is a hyperbolic group and the metric $\mathfrak{g}$ lifts to a metric $d_\mathfrak{g}$ on $\G$. Indeed, we consider the universal cover $(\widetilde{M},\widetilde{d_{\mathfrak{g}}})$ of $M$, fix a base point $ p \in \widetilde{M}$ and define
\[
d_\mathfrak{g}(x,y) = \widetilde{d_{\mathfrak{g}}}(x \cdot p, y\cdot p) \ \text{ for $x,y \in \G$.}
\]
Then the marked length spectrum of $\mathfrak{g}$ can be obtained as
\[
\ell_\mathfrak{g}[x] = \lim_{n\to\infty} \frac{d_\mathfrak{g}(o,x^n)}{n} \ \text{ for each $[x] \in \conj(\G)$}
\]
where $o \in \G$ is the identity element.
The metric $d_\mathfrak{g}$ on $\G$ may not be geodesic but roughly geodesic (see Section \ref{sec.prelim}).
Recall that for  $D \ge 0$ a $D$-rough geodesic for $d_\mathfrak{g}$ is a (possibly infinite) sequence of group elements $x_n \in \G$ such that 
\[
|m-n| - D \le d_\mathfrak{g}(x_n,x_m) \le  |m-n| + D \ \text{ for all $m,n \ge 1$.}
\]

There are Patterson-Sullivan measures associated to $d_\mathfrak{g}$ on the Gromov boundary $\partial \G$ of $\G$ (see Section \ref{sec.hyp}) and we label one of these measures $\nu_\mathfrak{g}$. A useful example to keep in mind is that when $(M,\mathfrak{g})$ is a closed, $n$ dimensional hyperbolic manifold. In this case the Gromov boundary of $\G$ is homeomorphic to the $n-1$ dimensional sphere and the Patterson-Sullivan measure $\nu_\mathfrak{g}$ is in the same class as the Lebesgue measure.

Our main result is as follows.

\begin{theorem}\label{thm.ae}
Let $(M,\mathfrak{g})$ be a compact Riemannian manifold with strictly negative sectional curvatures. Let $d_\mathfrak{g}$ be the metric introduced above with Patterson-Sullivan measure $\nu_\mathfrak{g}$. Then there exists $D >0$ such that for $\nu_\mathfrak{g}$ almost every $\xi \in \partial \G$ there exists a $D$-rough geodesic ray $(\xi_k)_{k=0}^\infty$ for $d_\mathfrak{g}$ starting at the identity in $\G$ ($\x_0 = o \in \G$) and with end point $\xi$ such that  for any $M\ge 1$ the set $
\{[\xi_k]: k\ge M\}$ is spectrally rigid.

Furthermore, for any function $f: \R_{>0} \to \R_{>0}$ with $f(T) \to \infty$ as $T\to\infty$ there exists a subsequence $n_k$ such that
$E= \{ [\xi_{n_k}] : k\ge 1\}$
is spectrally rigid and
\[
\#\{ [x] \in E :  \ell_\mathfrak{g}[x] < T  \} \le f(T) 
\]
for all $T >0$.
\end{theorem}
\begin{remark}
In fact the sequence $\x_k$ constructed in the above theorem will satisfy $|\ell_\mathfrak{g}[\x_k] - d_\mathfrak{g}(o,\x_k)| \le R \ \text{for all} \ k \ge 1$ and so $\ell_{\mathfrak{g}}[\x_k]$ is diverging linearly as $k$ increases.\\
\end{remark}

This result shows that there are arbitrarily sparse spectrally rigid sets that exist in $\nu_\mathfrak{g}$ almost ever every direction in $\G$. 
We immediately obtain the following corollary.
\begin{corollary}\label{thm.as}
Let $(M,\mathfrak{g})$ be a compact Riemannian manifold with strictly negative sectional curvatures. Then for any function $f:\R_{>0} \to \R_{> 0}$ with $f(T) \to \infty$ as $T\to\infty$ there exists a set $E \subset \conj(\G)$ (where $\G = \pi_1(M)$) such that
\[
\#\{ [x] \in E : \ell_\mathfrak{g}[x] < T\} \le f(T) \ \text{ for all } \ T > 0
\]
and $E$ is spectrally rigid. 
\end{corollary}
This shows that although there are no finite spectrally rigid sets, there are  arbitrarily sparse (with respect to $\mathfrak{g}$), infinite, spectrally rigid sets. \\

In the final section we prove an analogue of Theorem \ref{thm.ae} in which $d_\mathfrak{g}$ is replaced by a word metric.
This provides a refinement (in multiple ways, see Theorem \ref{thm.os}) of Kapovich's aforementioned result \cite{kap}. In particular it extends Kapovich's result to small actions of free groups and surface groups on $\R$-trees. We also prove a rigidity result for Hitchin representations, Theorem \ref{thm.hitchin}. In fact our methods are sufficiently flexible that they apply to a large class of metrics (including non-proper metrics) on general hyperbolic groups, see Theorem \ref{thm.general}. Before describing the format of the article we explain the methods and ideas behind the proof of Theorem \ref{thm.ae} in more detail.

\subsection{Ideas behind the proof}
There are two main ingredients in the proof of Theorem \ref{thm.ae}. The first is a dynamical description of the Patterson-Sullivan measure $\nu_\mathfrak{g}$. More precise we want to find a correspondence between  $\nu_\mathfrak{g}$ on $\partial \G$ and a measure which lives on a dynamical system which codes our hyperbolic group $\G$. This dynamical system, which is called a Cannon coding, is a subshift of finite type coming from a strongly Markov structure: a combinatorial object that represents the group $\G$ equipped with a finite generating set $S$ (see Section \ref{sec.cc}). In Section \ref{sec.ps} we show how to construct Patterson-Sullivan measures in a way that is compatible with the Cannon coding. This construction allows us to find a measure $\widehat{\nu}$ on the Cannon coding that represents $\nu_\mathfrak{g}$ on $\partial \G$.  We prove that this measure $\widehat{\nu}$ can be expressed in terms of certain Gibbs measures (Lemma \ref{lem.ac}): measures that are well understood from the perspective of ergodic theory. This allows us to form a dictionary between the geometric properties of the Patterson-Sullivan measure $\nu_\mathfrak{g}$ and the dynamical properties of $\widehat{\nu}$. This idea was inspired by the work of Calegari and Fujiwara.  In \cite{CalegariFujiwara2010} these authors studied the correspondence between dynamical measures on the Cannon coding and Patterson-Sullivan measures for \textit{word metrics}. The Cannon coding, as mentioned above, encodes $\G$ equipped with a word metric $S$. In particular the points in the Cannon coding represent infinite geodesic rays in the $S$ word metric.  The fact that the coding reflects the $S$ word metric  means that the Cannon coding is well suited for studying the Patterson-Sullivan measure  $\nu_S$  for a word metric. Calegari and Fujiwara use this to prove that there is a strong link between  $\nu_S$  and a dynamical measure called a measure of maximal entropy. In particular there is an explicit way of expressing $\nu_S$ as the push-forward of a measure of maximal entropy on the coding: see Proposition \ref{prop.mme} below.
In this work, we want to study the metrics $d_\mathfrak{g}$ opposed to word metrics. The Cannon coding is not as well suited for this purpose and significant technical difficulties appear that are not present when studying word metrics. 
The author is aware of only a couple of previous works that consider the links between Patterson-Sullivan measures for non-word metrics and dynamical measures on the Cannon coding. These works are by Gou\"{e}zel \cite{GouezelLocalLimit} and Tanaka \cite{THaus}. To prove Theorem \ref{thm.ae} we need to prove stronger comparison results between dynamical and geometrical measures than those shown in \cite{GouezelLocalLimit}, \cite{THaus}. To do so we appeal to \cite{cantrell.new} and \cite{cantrell.tanaka.2}. These works provide various structural results (see Proposition \ref{prop.mc} and Proposition \ref{prop.loops}) for the Cannon coding that allow us to form a correspondence between $\nu_\mathfrak{g}$ and Gibbs measures as mentioned above. The link between the Gibbs measures and $\nu_\mathfrak{g}$ is not as strong as the link between the measure of maximal entropy and $\nu_S$ (shown by Calegari and Fujiwara) and we do not have a direct analogue of Proposition \ref{prop.mme}. We are however able to prove a weaker version of Calegari and Fujiwara's result which is strong enough for our purposes (Lemma \ref{lem.ac}). Once we have constructed the dynamical measure $\widehat{\nu}$ representing $\nu_\mathfrak{g}$ and have understood its properties (see Section \ref{sec.ps}) we use results from dynamics to prove that a $\widehat{\nu}$ typical path in the coding sees (in some sense) all the finite paths in the coding. This is a dynamical analogue of a fact that Kapovich relies on to prove his result for sparse spectrally rigid sets in Outer Space \cite{kap}. Indeed, the proof of the main result  in \cite{kap} hinges on the fact that a (hitting measure) typical infinite path coming from a simple, non-backtracking random walk on a free group typically sees every group element as a subpath.
 
The second ingredient in the proof is a way of translating the recurrence properties of $\widehat{\nu}$ to the geometric setting to show that $\nu_\mathfrak{g}$ typical rough geodesics in $\G$ see all the conjugacy classes. We also need to do this in such a way that the resulting rough geodesics are spectrally rigid. This step occurs in Section \ref{sec.set} as Proposition \ref{prop.ssr}. To do so we use ideas from geometric group theory. This part of the argument still relies on understanding the Cannon coding however; we need to translate properties of points in the Cannon coding to properties of $d_\mathfrak{g}$ rough geodesics and to do so we need to show that the points in the coding represent particular nice geodesic representations for points in $\partial \G$. This second ingredient is different from Kapovich's approach in \cite{kap} which relies on geodesic currents. This is one of the key differences that allows us to generalise Kapovich's result.

\subsection{Organisation of the paper}\label{sec.organ}
In Section \ref{sec.prelim} we introduce background material on hyperbolic groups and symbolic dynamics. In the following section, Section \ref{sec.proof}, we prove Theorem \ref{thm.ae}. We break the proof into $3$ steps:
\begin{enumerate}
\item In Section \ref{sec.set},  we construct the sparse spectrally rigid sets that appear in the main theorem (see Definition \ref{def.v} and Proposition \ref{prop.ssr}). That is, we construct a set of $d_\mathfrak{g}$ rough geodesics and prove that it has the sparse spectrally rigid property in Theorem \ref{thm.ae}. We are then left to prove that this set has full $\nu_\mathfrak{g}$ measure.
\item In Section \ref{sec.ps} we present a dynamical construction of $\nu_\mathfrak{g}$ that we will use to prove that the set from Step $(1)$ has full measure. After constructing this dynamical measure we prove that it has various nice properties that we later rely on. 
\item We conclude the proof of Theorem \ref{thm.ae} in Section \ref{sec.conclude} using the results from the previous two steps.
\end{enumerate}
After proving Theorem \ref{thm.ae}, in the final section of the article we discuss generalisations of our main result as discussed above.

\subsection*{Acknowledgements}
The author is grateful to Eduardo Reyes, Richard Sharp, Ilya Kapovich and Andrey Gogolev for useful comments, discussions and remarks. The author is thankful to the referee for their thorough feedback and careful reading of this article; the author feels that this work has be much improved thanks to their comments.

%--------------------------------------------------------------------------------------------------------------------------------
 \section{Preliminaries}\label{sec.prelim}

\subsection{Hyperbolic groups and metrics}\label{sec.hyp}

We  will assume that the reader is familiar with hyperbolic groups and metrics and only briefly introduce them here. See \cite{GhysdelaHarpe} and \cite{cantrell.tanaka.1} for more detailed accounts.

Let $(X, d)$ be a metric space. The \textit{Gromov product} associated to $d$ is defined by
\[
 (x,y)_{w,d} =\frac{1}{2}( d(w, x)+d(w, y)-d(x, y)) \ \ \text{for $x, y, w \in X$}
\]
and we say that $(X, d)$ is $\d$-hyperbolic for some $\delta \ge 0$ if
\[
(x,y)_{w,d} \ge \min\left\{(x,z)_{w,d}, (y,z)_{w,d}\right\}-\d \quad \text{for all $x, y, z, w \in X$}.
\]
A metric space is called \textit{hyperbolic} if it is $\delta$-hyperbolic for some $\delta \ge 0$.

Let $\G$ be a finitely generated group. For a finite generating set $S$ the corresponding word metric $d_S$ is defined as follows. For $x, y \in \G$
\[
d_S(x,y) = \min\{k \ge 0: x^{-1}y = s_1\cdots s_k, \text{ for some } s_1, \ldots, s_k \in S \cup S^{-1} \}.
\]
Throughout we will use $o \in \G$ to denote the identity element in $\G$. We will also write $|x|_S$ for $d_S(o,x)$.
A hyperbolic group is a finitely generated group $\G$ that, when equipped with a word metric $d_S$ associated to a finite generating set $S$,  becomes a hyperbolic metric space. 
In this work all hyperbolic groups will be assumed to be non-elementary: we assume that they do not contain a cyclic subgroup with finite index.
We will use the notation $\Cay(\G,S)$ to denote the Cayley graph of $\G$ with respect to $S$.	

We say that two metrics $d,d_\ast$ on $\G$ are {\it quasi-isometric} if there exist constants $L > 1$ and $C \ge 0$ such that
\[
L^{-1} \, d(x, y)-C \le d_\ast(x, y) \le L \, d(x, y)+C \ \text{ for all $x, y \in \G$}.
\]
If two metrics $d,d_\ast$ have the property that there exists $\tau, C >0$ such that $|\tau d(o,x) - d_\ast(o,x)| < C$ for all $x \in \G$ (i.e. after scaling they agree up to a uniformly bounded error) then we say that $d,d_\ast$ are \textit{roughly similar}.
Throughout this work,  $\Dc_\G$ will denote the set of metrics on $\G$ which are left-invariant, hyperbolic and quasi-isometric to some (equivalently, any) word metric in $\G$. This class contains many interesting examples of metrics, see \cite{cantrell.tanaka.1}. Pairs of metrics $d,d_\ast$ in $\Dc_\G$ have  coarsely comparable Gromov products:
there exist $\lambda >1, C>0$ such that
\begin{equation}\label{eq.gp1}
\lambda^{-1}(x,y)_{o,d} - C \le (x,y)_{o,d_\ast} \le \lambda(x,y)_{o,d} + C
\end{equation}
for all $x,y \in \G$ \cite[Prop. 15 (i), Ch. 5]{GhysdelaHarpe}. For $d\in \Dc_\G$ we write $\ell_d$ for the translation distance function for $d$, i.e. $\ell_d[x] = \lim_{n\to\infty} d(o,x^n)/n$ where $[x]$ is the conjugacy class containing $x$. To simplify notation, if $S$ is a finite generating set for $\G$ we will write $\ell_S$ for the translation length function for $d_S$. 
For each $d \in \Dc_\G$ there exists a constant $C >0$ such that if $x \in \G$ satisfies $d(o,x) - 2(x, x^{-1})_{o,d} \ge C$ then
\begin{equation}\label{eq.tlgp}
| \ell_d[x]  - (d(o,x) - 2(x, x^{-1})_{o,d}| \le C,
\end{equation}
see \cite[Lemma 3.2]{cantrell.tanaka.1}. This will be a crucial property that we use multiple times in the proof of Theorem \ref{thm.ae}.

Fix a metric $d$ in $\Dc_\G$.
Given an interval $I \subset \R$ and constants $L, D >0$ we say that a map $\g: I \to \G$ is an \textit{$(L, D)$-quasi-geodesic} if 
\[
L^{-1} \, |s-t|-D \le d(\gamma(s), \gamma(t)) \le L\, |s-t|+D \  \text{ for all $s, t \in I$},
\]
and a \textit{$D$-rough geodesic} if
\[
|s-t|-D \le d(\gamma(s), \gamma(t))\le |s-t|+D \ \text{ for all $s, t \in I$}.
\]
Geodesics are $0$-rough geodesics.
A metric space $(\G, d)$ is called $D$-{\it roughly geodesic} if for each pair $x, y \in \G$ we can find a $D$-rough geodesic joining $x$ to $y$. We say that $(\G,d)$ is \textit{roughly geodesic} if it is $D$-roughly geodesic for some $D \ge 0$.
A geodesic metric space is a $0$-roughly geodesic metric space. Every metric in $\Dc_\G$ is roughly geodesic \cite{BonkSchramm}.
Moreover the Morse Lemma holds for such metrics: if  $d \in\Dc_\G$ is $D_0$-roughly geodesic then for every $(L, D)$-quasi-geodesic $\g$ in $(\G, d)$ there is a $D_0$-rough geodesic $\g_0$ such that $\g$ and $\g_0$ are within Hausdorff distance $K$ (depending on $L, D, D_0, \d$ where $\d$ is a hyperbolicity constant for $d$).

The main case that we consider in this work is the following example.
\begin{example}\label{ex.manifold}
Let $(M,\mathfrak{g})$ be a compact Riemannian manifold (possibly with totally geodesic boundary) with strictly negative sectional curvatures. 
The fundamental group $\pi_1(M) = :\G$ of $M$ is a hyperbolic group. Further, $\G$ acts on the universal cover $(\widetilde{M}, \widetilde{d}_\mathfrak{g})$ of $(M,\mathfrak{g})$ and through this action the metric $\mathfrak{g}$ lifts to a metric $d_\mathfrak{g}$:
\[
\text{ for $x,y \in \G$ } d_\mathfrak{g}(x,y) = \widetilde{d}_\mathfrak{g}(x \cdot p, y \cdot p) \ \text{ where $p \in \widetilde{M}$ is a fixed basepoint.}
\]
 This metric is quasi-isometric to a (equivalently every) word metric $d_S$ on $\G$ by the \v{S}varc-Milnor Lemma. Further it is left invariant and hyperbolic and so belongs to $\Dc_\G$.
\end{example}

Hyperbolic groups can be compactified using their ideal boundary $\partial \G$ which consists of equivalence classes
of divergent sequences. Fix a reference metric $d \in \Dc_\G$. A sequence of group elements $\{x_n\}_{n=0}^\infty$ diverges
if $(x_n,x_m)_{o,d}$ diverges as $\min\{n, m\}$ tends to infinity. Two divergent sequences $\{x_n\}_{n=0}^\infty$ and $\{y_n\}_{n=0}^\infty$ are {\it equivalent} if 
$(x_n,y_m)_{o,d}$ diverges as $\min\{n, m\}$ tends to infinity. 
If $d \in \Dc_\G$ is $D$-roughly geodesic then for each $\xi$ in $\partial \G$ there exists a $D$-rough geodesic beginning at $o$ and ending at $\xi$.
%$\g:[0, \infty) \to \G$ such that $\g(0)=o$ and $\g(n) \to \x$ as $n \to \infty$.
	
The Gromov product extends to points in $\G \cup \partial \G$ for any $d\in\Dc_\G$. Fix $d\in\Dc_\G$ and
let
\[
(\xi,\eta)_{o,d} = \sup\left\{\liminf_{n \to \infty}(x_n,y_n)_{o,d} \ : \ \x=\{x_n\}_{n=0}^\infty, \  \y=\{y_n\}_{n=0}^\infty\right\},
\]
for $\x, \y \in \G \cup \partial \G$,
where if $\x$ or $\y$ is in $\G$, then we take $\{x_n\}_{n=0}^\infty$ as the constant sequence $x_n=\x$ for all $n\ge 0$.
The Busemann function $ \beta_w(x, \x)$ associated to $d$ based at $w$ is given by
\[
\beta_w(x,\x) = \sup\left\{ \limsup_{n\to\infty} d(x,\x_n) - d(w,\xi_n) : \{ \x_n\}_{n=0}^\infty = \x \right\}.
\]
Following the Patterson-Sullivan construction one can obtain the following result (see for example Proposition 2.7 in \cite{cantrell.tanaka.1}). In the following and throughout the rest of this work we wiil write
\[
v_d = \limsup_{T\to\infty} \frac{1}{T} \log  \#\{ x \in \G: d(o,x) <T\}
\]
for the exponential growth rate of $d$.
\begin{proposition}\label{prop.qcm}
For each $d \in \Dc_\G$ there is a measure $\mu_d$ on $\partial \G$ which is ergodic with respect to the action of $\G$ on $\partial \G$. Furthermore there exists a constant $C > 1$ such that for any $x\in\G$
\[
C^{-1}\exp(-v_d\beta_{o}(x,\xi) ) \le  \frac{dx_\ast\mu_d}{d\mu_d} \le C\exp(-v_d\beta_{o}(x,\xi))
\]
where $\b_o$ is the Busemann functions for $d$.
\end{proposition}
Here ergodic means that for any $U,V \subset \G$ with $\mu_d(U), \mu_d(V) > 0$ there exists $x\in\G$ such that $\mu(x(U) \cap V) > 0$.

%-------------------------------------------------------------------------------------------------------------------------------------------------------------------

\subsection{Cannon coding and thermodynamic formalism}\label{sec.cc}

Fix a finite generating set $S$ for $\G$.

\begin{definition} \label{def.sms}
Let  $\Ac=(\Gc, w, S)$ be a triple where
\begin{enumerate}
\item $S$ is a finite generating set for $\G$;
\item $\Gc=(V, E, \ast)$ is a finite directed graph with a vertex $\ast$ called the \textit{initial state}; and,
\item $w: E \to S$ is a labelling such that
for a directed edge path $(x_0, x_1, \dots, x_{n})$ (where $(x_i, x_{i+1})$ corresponds to a directed edge)
there is an associated path in $\Cay(\G, S)$ beginning at the identity. This path corresponds to
\[
(o, w(x_0,x_1), w(x_0,x_1)w(x_1,x_2), \dots, w(x_0,x_1) \, \cdots \, w(x_{n-1}, x_n)).
\]
\end{enumerate}
Let $\ev$ denote the map that sends a finite path to the group element obtained by multiplying the labelings along the path, i.e.
$\ev(x_0,\ldots, x_n) = w(x_0,x_1)\cdots \, w(x_{n-1},x_n)$.
We say that $\mathcal{A}$  is a \textit{strongly Markov structure} for $\G$ and $S$ if 
\begin{enumerate}
\item for each vertex $v\in V$ there exists a directed path from $\ast$ to $v$;
\item for each directed path in $\Gc$ the associated path in $\Cay(\G, S)$ is a geodesic; and,
\item the map $\ev$ defines a bijection between the set of directed paths from $\ast$ in $\Gc$ and $\G$.
\end{enumerate}
\end{definition}

Cannon  proved that some cocompact Kleinian groups and generating sets admit strongly Markov structures \cite{Cannon}. Ghys and de la Harpe extended Cannon's result to every hyperbolic group and finite generating set \cite{GhysdelaHarpe}. See also \cite[Section 3.2]{Calegari}. 

For technical reasons we augment the strongly Markov structure by introducing an additional vertex labelled $0$ and add directed edges from every vertex $x \in V \cup \{0\} \backslash \{\ast\}$ to $0$ and define $\o(x,0) = o$ (the identity in $\Gamma$) for every $x \in V \cup \{0\} \backslash \{\ast\}$. We do this so that it is possible to see both the group elements in $\G$ and the geodesic rays in $\partial \G$ as infinite paths in the Cannon coding. Indeed, we can see the elements in $\partial \G$ as infinite paths in the Cannon coding that do not visit the $0$ vertex and the group elements in $\G$ as the infinite paths that visit 0.  We will assume that every strongly Markov structure has been augmented in this way and will abuse notation by labelling the augmented structure, its edge and vertex set by $\mathcal{G}$, $V$ and $E$ respectively. We then extend the map $\ev$ defined on finite paths in Definition \ref{def.sms} to infinite paths.
If $x = (x_k)_{k=0}^\infty$ is an infinite path in $\mathcal{G}$ and $n \in \Z_{\ge 0}$ then we set
\[
\ev_n(x) = \ev(x_0, x_1, \ldots, x_n).
\]
We also define $\ev(x)$ to be the point in $\G \cup \partial \G$ corresponding to the geodesic (which is either finite if it ends in an infinite string of $0$s or infinite otherwise) determined by $x$.

The strongly Markov structure $\mathcal{G}$ allows us to introduce a subshift of finite type. We will now recall some of the preliminary results we need from symbolic dynamics. See \cite{ParryPollicott} for a more thorough introduction. Fix a finite directed graph $\mathcal{G}$ and let $A$ be the $k \times k$ (where $k$ is the number of vertices in $\Gc$), $0-1$ transition matrix describing $\Gc$. 
We use the notation $A_{i,j}$ to denote the $(i,j)$th entry of $A$. The one-sided subshift of finite type associated to $A$ is the space
\[
\Sigma_A = \{(x_n)_{n=0}^{\infty} : x_n \in \{1,2,...,k\}, A_{x_n,x_{n+1}}=1, n \in \mathbb{Z}_{\ge 0}\}.
\]
When $A$ is clear, we will drop $A$ from the notation in the definition of the shift spaces and will simply write $\SS$. Given $x\in\SS_A$ we write $x_n$ for the $n$th coordinate of $x$. 
The shift map $\sigma: \Sigma_A \rightarrow \Sigma_A$ sends $x$ to $y= \sigma(x)$ where $y_n=x_{n+1}$ for all $n \in \mathbb{Z}_{\ge 0}$.  Consider a finite ordered string $x_0, \ldots, x_{m-1} \in \{1, \ldots, k \}$ where $A_{x_j, x_{j+1}} =1$ for each $j=0, \ldots, m-2$. The cylinder set associated to this string is the subset of $\Sigma_A$ given by
\[
[x_0, \ldots, x_{m-1}] = \left\{ (y_n)_{n=0}^\infty \in \Sigma_A : y_j = x_j \text{ for } j=0, \ldots, m-1 \right\}.
\]
We endow $\Sigma_A$ with a topology by declaring the set of all cylinder sets to be an open basis.
For each $0<\theta <1$ we can define a metric $d_\theta$ on $\Sigma_A$ which is compatible with this topology. Let $x,y \in \Sigma_A$.  If $x_0 = y_0$ then we define $d_\theta(x,y)= \theta^N,$ where $N$ is the largest positive integer such that $x_i = y_i$ for all $0 \le i<N$. If $x_0 \neq y_0$ we set $d_\theta(x,y)=1$. We let
\[
F_\theta(\Sigma_A) = \{r:\Sigma_A \rightarrow \mathbb{C}:\text{$r$ is Lipschitz with respect to $d_\theta$}\}.
\]
which we equip with the norm $\|r\|_\theta = |r|_\theta + |r|_\infty$ where $|r|_\infty$ is the sup-norm and $|r|_\theta$ denotes the least Lipschitz constant for $r$. This space is a Banach space. We say that a function $f : \Sigma_A \to \R$ is \textit{H\"older} if it belongs to $F_\theta(\Sigma_A)$ for some $0<\theta <1$.

Given a strongly Markov structure $\Gc$ associated to a pair $(\G, S)$ we call the corresponding subshift of finite type a \textit{Cannon coding}. We would like to understand the dynamical properties of Cannon codings. These properties are linked to the connectedness properties of the corresponding strongly automatic structure. 

We say that a directed graph $\Gc$ is connected if each ordered pair of vertices in $\Gc$ are connected by a path.
A \textit{connected component} of a finite directed graph is a maximal, connected subgraph. Given a strongly automatic structure $\Gc$ associated to a pair $(\G, S)$ the $\ast$ state only has outgoing edges and so $\Gc$ will never be connected. We can however decompose $\Gc$ into connected components. % Given such a connected component $\Cc$ there exists a maximal integer $p_\Cc \ge 1$ such that the length of every closed loop in $\Cc$ has length divisible by $p_\Cc$. This constant $p_\Cc$ is called the period of $\Cc$.
%When $p_\Cc =1$ we say that the component $\Cc$ is aperiodic and when this is the case the corresponding subshift $\SS_\Cc$ defined over $\Cc$ is mixing. In general $\Cc$ is connected  and $(\SS_\Cc, \sigma)$ is topologically transitive (but not necessarily weak mixing). In particular, for any two non-empty open sets $U$ and $V$, there exists $n \ge 1$ such that $U\cap \sigma^{-n} V \neq \emptyset$.
%When we have that $p_\Cc >1$ we can decompose the vertex set $V(\Cc)$ for $\Cc$ into a disjoint collection of vertices $V_1, \ldots, V_{p_\Cc}$, i.e. $V(\Cc) = \bigsqcup_{j \in \Z/p_\Cc\Z} V_j$.  For $j=1, \ldots, p_\Cc$ we use $\Sigma_j$ to denote the set of elements in $\Sigma_\Cc$ that contain sequences starting with a vertex belonging to $V_j$.  Then $\sigma(\Sigma_j) = \Sigma_{j+1}$ where $j,j+1$ are taken modulo $p_\Cc$. Furthermore, each system $(\SS_j, \sigma^{p_\Cc})$ is a mixing subshift of finite type. 
We can apply results from thermodynamic formalism to the subshifts defined over connected components. For example, for each component $\Cc$, the variational principle holds. Let $\mathcal{M}(\sigma, \Sigma_\Cc)$ denote the collection of $\sigma$ invariant probability measures on $\Sigma_\Cc$.

\begin{proposition}[Theorem 3.5 \cite{ParryPollicott}]\label{prop.vp}
If $\P$ is a H\"older continuous function on $\Sigma_\Cc$ then the following supremum
\[
\textnormal{P}_\Cc(\P) =\sup_{\lambda \in \Mcc(\sigma, \Sigma_\Cc)}\left\{h(\s, \lambda)+\int_{\Sigma_\Cc}\P\,d\lambda\right\}
\]
is attained by a unique $\sigma$-invariant probability measure $\m_{\Cc,\P}$ on $\Sigma_\Cc$. Further $\mu_{\Cc,\P}$ satisfies the Gibbs property: there exists a positive constant $c > 1$ such that
\begin{equation}\label{eq:gibbs}
c^{-1} \le \frac{\m_{\Cc,\P}[x_0, \dots, x_{n-1}]}{\exp\(-n\textnormal{P}_\Cc(\P)+ \P^n(x)\)} \le c
\end{equation}
for all $x \in [x_1, \dots, x_{n-1}]$ and for all $n \in \Z_{\ge 1}$.
The quantity $\textnormal{P}_\Cc(\P)$ is referred to as the \textit{pressure} of $\P$ over $\Cc$.
\end{proposition}
When we are considering a Cannon coding $\Sigma$ with multiple connected components and $\P \in F_\theta(\SS)$, we will write
\[
\text{P}(\P)=\max_\Cc \text{P}_\Cc(\P),
\]
where $\Cc$ runs over all components in $\Gc$.

We then have the following result, in which, given a function $r\in F_\theta(\Sigma_A)$ and $n\ge 1$ we write $ r^n$ for the $n$th Birkhoff sum of $r$, i.e. $r^n(x) = r(x) + r(\sigma(x)) + \cdots + r(\sigma^{n-1}(x))$.

\begin{lemma}\label{lem.holder}[Example 4.11 \cite{cantrell.tanaka.1} and Lemma 4.8 \cite{cantrell.tanaka.1}]
Let $d_\mathfrak{g}$ be the metric constructed in Example \ref{ex.manifold}.
Then, for any Cannon coding $\SS$, we can find a H\"older continuous potential $\P_\mathfrak{g} $ such that
\[
\P^n_\mathfrak{g}(x)=\sum_{i=0}^{n-1}\P_\mathfrak{g}(\s^i(x))= d_\mathfrak{g}(o, \ev_n(x))+O(1) \quad \text{for all $x \in \Sigma$}
\]
uniformly in $x, n$. In fact, we may take $\P_\mathfrak{g}(x) = \beta_{o,\mathfrak{g}}(\ev_1(x),\ev(x))$ (the Busseman cocycle for $d_\mathfrak{g}$).
Furthermore, if $v_{\mathfrak{g}}$ is the exponential growth rate of $d_\mathfrak{g}$ then $\textnormal{P}(-v_\mathfrak{g}\P_\mathfrak{g}) = 0$. 
\end{lemma}
We also have the following well-known result that follows from the Ruelle-Perron-Frobenius Theorem \cite{ParryPollicott}.
\begin{lemma}\label{lem.rpf}
Let $\P_\mathfrak{g}$ be as in Lemma \ref{lem.holder}.  If $\textnormal{P}_\Cc(-v_\mathfrak{g}\P_\mathfrak{g}) = 0$ and $y \in \Sigma_\Cc$ is fixed, there exists $C > 1$ such
\[
C^{-1} \le \sum_{\sigma^n(x) = y, x \in \Sigma_\Cc} \exp(-v_\mathfrak{g}\P_\mathfrak{g}^n(x)) \le C \ \text{ for all $n\ge 1$}
\]
where the sum is over all $x\in\Sigma_\Cc$ for $\Cc$ with $\sigma^n(x) = y$. If $\exp\{\textnormal{P}_\Cc(-v_\mathfrak{g}\P_\mathfrak{g})\} < \theta <1$ and $y \in \Sigma_\Cc$ then there exists $C'>0$ such that
\[
\sum_{\sigma^n(x) = y, x \in \Sigma_\Cc} \exp(-v_\mathfrak{g}\P_\mathfrak{g}^n(x)) \le C' \theta^n \ \text{ for all $n\ge 1$}.
\]
\end{lemma}

\begin{definition}
We call a component $\Cc$ \textit{$-v_{\mathfrak{g}} \P_\mathfrak{g}$-maximal} if the pressure of $-v_\mathfrak{g} \P_\mathfrak{g}$ over the component $\Cc$ is equal to $0$. We call a component $\Cc$ \textit{word maximal} if the number of closed loops in $\Cc$ of length $n$ has exponential growth rate $v_{d_S}$. 
\end{definition}

Since the metric $d_\mathfrak{g}$ is different (i.e. not roughly similar) to any word metric $d_S$, it appears possible for their to be a $-\P_\mathfrak{g}$-maximal component that is not a word maximal component. However, somewhat surprisingly, this cannot happen.
\begin{proposition}\cite[Proposition 3.1]{cantrell.new}\label{prop.mc}
Fix a Cannon coding for $(\G,S)$. Then the  $-v_{\mathfrak{g}} \P_\mathfrak{g}$ maximal components are precisely the word maximal components.
\end{proposition}
\noindent We will use this result implicitly throughout this work.  We also note that word maximal (equivalently $-v_{\mathfrak{g}} \P_\mathfrak{g}$ maximal components) are disjoint: there does not exists a path connecting vertices in distinct word maximal components \cite[Lemma 4.10]{CalegariFujiwara2010}. Throughout this work, when we have a Cannon coding $\Sigma$, we will label the word maximal components $B_1, \ldots, B_k$.

\begin{definition}
Define the set  $\Sigma_{[\ast]} \subset \Sigma_{A}$ by 
\[
 \Sigma_A = \{ x = (x_n)_{n=0}^\infty \in \Sigma_{A} : x_0 =\ast\}.
 \]
Let $\ev_\ast: \Sigma_{[\ast]} \to \G \cup \partial \G$ be the natural map associated to the bijection defined in Definition $2.2$. 
\end{definition}
We end this section by recording the following result that will be crucial in our proofs.

\begin{proposition}\cite[Corollary 3.6]{cantrell.new} \label{prop.loops}
Suppose that $\mathcal{C}$ is a word maximal component in a Cannon coding for $(\G, S)$. Then there exists $M >0$ such that for any non-torsion conjugacy class $[g] \in \conj(\G)$ there is a periodic orbit $x =(x_0, x_1, \ldots, x_l, x_0, \ldots ) \in \Sigma_\mathcal{C}$ (for some $l >0$) such that $\ev(x_0, \ldots, x_l,x_0)$ belongs to one of $[g^{\pm M}]$.
\end{proposition}

%%%%%%%%%%%%%%%%%%%%%%%%%%%%%%%%%%%%%%%%%%%%%%%%%%%%%%%%%%%%%%%%%%%%%%%%%%%%%%%%%%%%%%%%%%%%%%%%%%%%%%%%%%%%%%%%%%%%%%%%%%%%%%%%%%%%

\section{Proof of the main theorem}\label{sec.proof}
In this section we prove Theorem \ref{thm.ae} following the steps outlined in Section \ref{sec.organ}.
Fix a compact, negatively curved Riemannian manifold $(M,\mathfrak{g})$ and finite generating set $S$ for $\G = \pi_1(M)$.
Choose and fix a Cannon coding $\Sigma$ for $(\G,S)$ and let $\ev_\ast: \Sigma_{[\ast]} \to \G \cup \partial \G$ be the map from the previous section. Write $B_1, \ldots, B_k$ for the word maximal components in $\Sigma$. 

%%%%%%%%%%%%%%%%%%%%%%%%%%%%%%%%%%%%%%%%%%%%%%%%%%%%%%%%%%%%%%%%%%%%%%%%%%%%%%%%%%%%%%%%%%%%%%%%%%%%%%%%%%%%%%%%%%%%%%%%%%%%%%%%%%%%

\subsection{Identifying the sparse spectrally rigid sets}\label{sec.set}
We first construct the sparse spectrally rigid sets appearing in Theorem \ref{thm.ae}.
For each word maximal component $B_1, \ldots, B_k$  define
\[
\widehat{U}_j = \{ x \in \Sigma_{B_j} : x \text{ visits every cylinder in $\Sigma_{B_j}$ infinitely often} \} \ \text{ and set } \ 
\widehat{U} = \bigcup_{j=1}^k \widehat{U}_j.
\]
We then let
\[
 U = \Sigma_{[\ast]} \cap \bigcup_{l \ge 1}  \sigma^{-l}\left(\widehat{U}\right)
 \]
and make the following definition.
\begin{definition} \label{def.v}
Define $V \subset \partial \G$ to be given by $V = \ev_\ast(U)$.
\end{definition}
We will prove that $V$ defines  the set of typical rough geodesics appearing in Theorem \ref{thm.ae}. We begin  by showing  that the points in $V$ have  $d_\mathfrak{g}$ rough geodesic representatives that are spectrally rigid.

\begin{proposition}\label{prop.ssr}
There exists $D >0$ such that for every $\xi \in V$ there exists a $D$-rough geodesic ray $(\xi_k)_{k=0}^\infty$ for $d_\mathfrak{g}$ starting at the identity in $\G$ ($\x_0 = o \in \G$) and with end point $\xi$ such that  for any $M\ge 1$ the set $
\{[\xi_k]: k\ge M\}$ is spectrally rigid.

Furthermore, for any function $f: \R_{>0} \to \R_{>0}$ with $f(T) \to \infty$ as $T\to\infty$ there exists a subsequence $n_k$ such that
$E= \{ [\xi_{n_k}] : k\ge 1\}$
is spectrally rigid and
\[
\#\{ [x] \in E :  \ell_\mathfrak{g}[x] < T  \} \le f(T) 
\]
for all $T>0$.
\end{proposition}

\begin{proof}
Take $\xi \in V$. Then, by construction there exists $z=(z_n)_{n=0}^\infty \in U$ such that $\ev_\ast(z) = \xi$. 
Since $z$ corresponds to a $d_S$ geodesic ray with endpoint $\x$, it also represents a $(\lambda,C)$-quasi-geodesic ray for $d_\mathfrak{g}$ for some $L,C >0$ depending only on $\mathfrak{g}$ and $S$. By the Morse Lemma there exists $D' > 0$ (depending only on $\mathfrak{g}$ and $S$) and a $D'$-rough geodesic ray for $d_\mathfrak{g}$ that fellow travels the $d_S$ geodesic $\{\ev_n(z)\}_{n\ge1}$. This rough geodesic, which we label $\{g'_k\}_{k \ge 1}$ remains within Hausdorff distance $K' >0$ (depending only on $d_\mathfrak{g}$ and $S$) of $\{\ev_n(z)\}_{n\ge1}$.

Now, by \cite[Lemma 3.11]{cantrell.reyes.man} %to construct a $D$-rough geodesic ray (for $d_S$) that fellow travels $\x$. By this result 
there exist $D_1, R_1 >0$ such that for each $n\ge 1$ we can find a group element $g_n$ close to $g'_n$ such that
\begin{equation} \label{eq.gp}
d_S(g_n,g'_n) \le D_1, \ (g_n,g_n^{-1})_S \le D_1 \ \text{ and } \ |\ell_S[g_n] - |g_n|_S|\le R_1.
\end{equation}
It follows that $\{g_n\}_{n\ge 1}$ is a $D$-rough geodesic ray (for some $D>0$ depending only on $\mathfrak{g}$, $S)$ that we can choose to  start at $o$ and that ends at $\x = \ev_\ast(z) \in \partial \G$. Further there exists $K >0$ such that $\{ g_n \}_{n\ge 1}$ remains within Hausdorff distance $K$ of $\{\ev_n(z)\}_{n\ge1}$. We will show that $\{ [g_n] : n\ge 1\}$ is spectrally rigid. A similar proof shows that $\{ [g_n] : n\ge M\}$ is spectrally rigid for any fixed $M \ge 1$.

By construction $z$ spends its tail in one of the maximal components, say $B_i$. Suppose that $z$ enters $B_i$ on step $k$ (i.e. $\sigma^k(z) \in \Sigma_{B_i}$) and then stays in there forever. Further $z$ visits each cylinder in $B_i$ infinitely often.  
We will leverage this fact to show that  $\{ [g_n] : n\ge 1\}$ is spectrally rigid. Suppose that $\mathfrak{g}_1, \mathfrak{g}_2$ are two Riemannian metrics on $\G$ and write $\ell_1, \ell_2$ for the corresponding translation distance functions. We will show that, if $\ell_1, \ell_2$ agree on $\{ [g_n] : n\ge 1\}$ then there exists a constant $C >0$ (independent of $[x]$) such that
\[
|\ell_1[x] - \ell_2[x]| \le C \ \text{ for each } \ [x] \in \conj(\G).
\]
This implies that $\ell_1[x] = \ell_2[x]$ everywhere. 

Suppose that $\ell_1, \ell_2$ agree on conjugacy classes in $\{ [g_n] : n\ge 1\}$.
Now, fix a conjugacy class $[x] \in \conj(\G)$. By Proposition \ref{prop.loops} there exists an integer $M \ge 1$ such that there is a loop in  $B_i$ that represents $[x^{\pm M}]$. Suppose that this loop is described by the path $r_1, \ldots, r_l, r_1$ so that $\ev(r_1, \ldots, r_l, r_1) \in [x^{\pm M}]$. 
Since $z$ visits every cylinder infinitely often in $B_i$ (and in particular visits $[r_1,\ldots, r_l, r_1]$ infinitely often) we deduce that $z$ belongs to a cylinder of the form
\[
[ \ast = z_0, \ldots, z_k, p_1, \ldots, p_n, r_1, \ldots, r_l, r_1]
\]
for some path $p_1, \ldots, p_n$ in $B_i$.
%%%%%%%%%%%%%%%%%%%%%%%%%%%%%%%%%%%%%%%%%%%%%%%%%%%%%%%%

We define $h_1 = \ev(z_0,\ldots, z_k,p_1,\ldots, p_n, r_1)$, $h_2 = \ev(z_0,\ldots, z_k,p_1,\ldots, p_n, r_1, \ldots, r_l, r_1) $ and let $d_1, d_2 \in \Dc_\G$ be the lifts of the metrics $\mathfrak{g}_1, \mathfrak{g}_2$ to $\G = \pi_1(M)$ as in Example \ref{ex.manifold}. Then note that, by the Morse Lemma
\[
d_j(o, h_2 ) = d_j(o,h_1) + d_j(o,\ev(r_1,\ldots,r_l,r_1)) + O(1) \ \text{ for $j=1,2$}
\]
where the implied constant is independent of $[x]$. 
Now, since $\{ g_n \}_{ n\ge 1}$ is within Hausdorff distance $K >0$ of $\{\ev_n(z)\}_{n\ge1}$ we deduce that there are $N_1, N_2 > 0$ such that
\[
d_\mathfrak{g}(h_1 , g_{N_1}) \le K \ \text{ and } \ d_\mathfrak{g}(h_2 , g_{N_2}) \le K.
\]
%where $h_1 = \ev(z_0,\ldots, z_k,p_1,\ldots, p_n, r_1)$ and $h_2 = \ev(z_0,\ldots, z_k,p_1,\ldots, p_n, r_1, \ldots, r_l, r_1) $.
%Now, we lift the metrics $\mathfrak{g}_1$ and $\mathfrak{g}_2$ to $\G = \pi_1(M)$ to obtain metrics $d_1, d_2 \in \Dc_\G$ as in Example \ref{ex.manifold}.
Further, by the Gromov product condition in (\ref{eq.gp}), equation (\ref{eq.tlgp}) and the fact that $d_1$ and $d_2$ have Gromov products comparable with that for $d_S$ by (\ref{eq.gp1}), we also have that, for $j=1, 2$,
\[
d_j(o, h_1) = \ell_j[g_{N_1}] + O(1) \ \text{ and } \ d_j(o, h_2) = \ell_j[g_{N_2}] + O(1)
\]
where the implied constants are independent of $[x]$. Also, since $r_1, \ldots, r_l, r_1$ corresponds to a loop in the Cannon coding, we have that
\[
(\ev(r_1, \ldots, r_l, r_1), \ev(r_1, \ldots, r_l, r_1)^{-1})_{o,d_S} = 0
\]
and so $d_j(o,\ev(r_1,\ldots,r_l,r_1)) = \ell_j[x^M] + O(1)$ by (\ref{eq.gp1}) and  (\ref{eq.tlgp}) where the error is independent of $[x]$.
% Further by the Morse Lemma,
%\[
%d_j(o, h_2 ) = d_j(o,h_1) + d_j(o,\ev(r_1,\ldots,r_l,r_1)) + O(1) \ \text{ for $j=1,2$}
%\]
%where the implied constant is independent of $[x]$. 
Combining these expressions show that
\[
\ell_j[ g_{N_1}]  = \ell_j[g_{N_2}] + \ell_j[x^M] + O(1) \ \text{ for $j=1,2$}
\]
where the implied constant is independent of $[x]$. However we also have by assumption that $\ell_1[g_{N_1}] = \ell_2[g_{N_1}]$ and $\ell_1[g_{N_2}] = \ell_2[g_{N_2}]$ and so we deduce that
\[
\ell_1[x^M] = \ell_1[g_{N_1}] - \ell_1[g_{N_2}] + O(1) = \ell_2[g_{N_1}] - \ell_2[g_{N_2}] + O(1) = \ell_2[x^M] + O(1)
\]
where the implied error does not depend on $[x]$. This concludes the proof of the first part of the proposition.

For the furthermore part we note that to deduce that $\ell_1[x^M]  = \ell_2[x^M] + O(1)$ we need only know that $\ell_1[g_{N_1}] = \ell_2[g_{N_1}]$ and $\ell_1[g_{N_2}] = \ell_2[g_{N_2}]$. That is, for any $[x] \in \conj(\G)$ we can find a pair of elements $g_{N_1}([x]), g_{N_2}([x])$ such that if $\ell_1, \ell_2$ agree on these elements then $\ell_1[x^M]  = \ell_2[x^M] + O(1)$ where the error is independent of $[x]$. Since $z$ visits every cylinder  in $B_i$ infinitely often we can take $N_1,N_2$ to be as large as we like. This observation allows us to inductively construct a spectrally rigid set with growth rate slower than any prescribed function. Fix a function $f: \R_{>0} \to \R_{>0}$ with $f(T) \to \infty$ as $T\to \infty$. We begin by listing the conjugacy class in $\G$, $[x_1], [x_2], \ldots$ and for $[x_1]$ we find a pair  $g_N[x_1], g_{N'}[x_1]$ of group elements with word lengths $N[x_1], N'[x_1]$. We can take $N[x_1], N'[x_1]$ sufficiently large so that when we consider the set $E_1$ containing this pair of elements we have that $\#\{ [x] \in E_1 : \ell_\mathfrak{g}[x] < T \} \le f(T)$ for all $T >0$. We then do the same for $[x_2]$ but choose $g_N[x_2], g_{N'}[x_2]$ so that their word lengths $N[x_2], N'[x_2]$ are sufficiently large so that the set $E_2 = E_1 \cup \{ g_N[x_2], g_{N'}[x_2]\}$ has the property that $\#\{ [x] \in E_2 : \ell_\mathfrak{g}[x] < T \} \le f(T)$ for all $T >0$. Continuing inductively we obtain a set $E$ such that $\#\{ [x] \in E : \ell_\mathfrak{g}[x] < T \} \le f(T)$ for all $T >0$ and $E$ is spectrally rigid.
\end{proof}

We are now left to show that $\nu_\mathfrak{g}(V) = 1$. To prove this assertion we first need to develop a dynamical construction for the Patterson-Sullivan measure which we discussed in the introduction.

%%%%%%%%%%%%%%%%%%%%%%%%%%%%%%%%%%%%%%%%%%%%%%%%%%%%%%%%%%%%%%%%%%%%%%%%%%%%%%%%%%%%%%%%%%%%%%%%%%%%%%%%%%%%%%%%%%%%%%%%%%%%%%%%%%%%

\subsection{A dynamical Patterson-Sullivan construction}\label{sec.ps}
In this section we show how to construct a Patterson-Sullivan measure for $d_\mathfrak{g}$ that is compatible with the Cannon coding $\Sigma$. As mentioned in the introduction, this construction is inspired by the work of Calegari and Fujiwara \cite{CalegariFujiwara2010}. Throughout this subsection we continue to use the notation established at the beginning of Section \ref{sec.proof}.

We begin by defining a sequence of measures
\[
\nu_n = \frac{\sum_{|x|_S \le n} \exp(-v_\mathfrak{g}d_\mathfrak{g}(o,x)) \delta_{\{x\}}}{\sum_{|x|_S \le n} \exp(-v_\mathfrak{g}d_\mathfrak{g}(o,x))}
\]
where $\d_{\{x\}}$ is the Dirac measure based at $x \in \G$. We can check that any weak star convergent subsequence $\nu$ of $\nu_n$ is supported on $\partial \G$ and belongs to the same measure class as the Patterson-Sullivan measures for $d_\mathfrak{g}$. Indeed, by direct computation we see that there exists $C>1$ such that for any $x \in \G$
\[
C^{-1} \exp(-v_\mathfrak{g} \beta_{o,\mathfrak{g}}(x,\x)) \le \frac{d x_\ast \n}{d\nu}(\x) \le C \exp(-v_\mathfrak{g} \beta_{o,\mathfrak{g}} (x,\x))
\]
where $\beta_{o,\mathfrak{g}}$ is the Busseman cocycle associated to $d_\mathfrak{g}$. This then implies $\nu$ is a Patterson-Sullivan measure for $d_\mathfrak{g}$ by \cite[Theorem 3.6]{THaus}.

We now consider the sequence of measures $\widehat{\nu}_n$ on the set of paths that start at $\ast$ in $\Sigma$ and visit the $0$ state (and so end in infinite repetition of $0$) that are defined by the relation $\nu_n(\{g\}) = \widehat{\nu}_n(\ev^{-1}\{g\})$ for $g \in \G$. Since $\ev$ is continuous we have the following.
\begin{lemma}\label{lem.wsc}
Any weak star convergent subsequence of $\widehat{\nu}_n$ pushes forward under $\ev_\ast$ to a Patterson-Sullivan measure for $d_\mathfrak{g}$ on $\partial \G$.
\end{lemma}
Fix a weak star limit $\widehat{\nu}$ of the sequence $\widehat{\nu}_n$ and note that $\widehat{\nu}$ is supported on $\Sigma_{[\ast]}$. We want to understand how $\widehat{\nu}$  assigns mass to cylinders. We know by definition that the value of $\widehat{\nu}[\ast, x_1, \ldots, x_{k-1}]$ is the limit along a subsequence of the following sequence
\begin{equation}\label{eq.ps}
\frac{\sum_{j=k}^n \sum_{x_{k-1}, y_1, \ldots, y_{n-j}} \exp(-v_\mathfrak{g}d_\mathfrak{g}(o,\ev(\ast, x_1,\ldots, x_{k-1}, y_1, \ldots, y_{n-j})))
}{\sum_{|x|_S \le n} \exp(-v_\mathfrak{g}d_\mathfrak{g}(o,x))}
\end{equation}
where the sum over $x_{k-1}, y_1, \ldots, y_{n-j}$ is over all $y_1, \ldots, y_{n-j}$ such that $x_{k-1}, y_1, \ldots, y_{n-j}$ is a path allowed by $A$. 
By Lemma 2.8 in \cite{cantrell.tanaka.1} there exists  $C > 1$ such that
\[
C^{-1} n \le \sum_{|x|_S \le n} \exp(-v_\mathfrak{g} d_{\mathfrak{g}}(o,x)) \le C n \ \ \text{ for each $n \ge 1$}
\]
and also, by the Morse Lemma (since $(\ast, x_1,\ldots, x_{k-1}, y_1, \ldots, y_{n-j})$ corresponds to a geodesic in the $S$ word metric), the distance $d_\mathfrak{g}(o,\ev(\ast, x_1,\ldots, x_{k-1}, y_1, \ldots, y_{n-j}))$ is equal to
\[
d_\mathfrak{g}(o,\ev(\ast, x_1,\ldots, x_{k-1}) + d_\mathfrak{g}(o,\ev(x_{k-1}, y_1, \ldots, y_{n-j}))
\]
up to a uniformly bounded additive constant independent of $k$ and $n$.
These inequalities imply that for each $n\ge 1$ the quantity $(\ref{eq.ps})$ is bounded above and below by uniformly bounded (away from $0$ and $\infty$) constants (depending only on $\mathfrak{g}$, $d_S$) multiplied by
\begin{equation}\label{eq.bound}
\exp(-v_\mathfrak{g}d_\mathfrak{g}(o, \ev(\ast, x_1, \ldots, x_{k-1}))) \left( \frac{1}{n} \sum_{j=k}^n \sum_{x_{k-1}, y_1, \ldots, y_{n-j}}  \exp(-v_\mathfrak{g} d_\mathfrak{g}(o,\ev(x_{k-1}, y_1, \ldots, y_{n-j})))\right).
\end{equation}
We then have the following. 
\begin{lemma}\label{lem.bounds}
If there is a path from $x_{k-1}$ into a word maximal component then there exists $C > 1$ such that
\[
C^{-1} \le \sum_{x_{k-1}, y_1, \ldots, y_{n-j}}  \exp(-v_\mathfrak{g} d_\mathfrak{g}(o,\ev(x_{k-1}, y_1, \ldots, y_{n-j}))) \le C
\]
for all $n,j \ge k$. If no such path exists then there are $0 < \theta < 1$ and $C' >0$ such that
\[
\sum_{x_{k-1}, y_1, \ldots, y_{n-j}}  \exp(-v_\mathfrak{g} d_\mathfrak{g}(o,\ev(x_{k-1}, y_1, \ldots, y_{n-j}))) \le C' \theta^{n-j}
\]
for all $n,j \ge k$.
\end{lemma}
In the proof of this lemma and moving forward, $\P_\mathfrak{g}$ will denote the H\"older continuous function associated to $d_\mathfrak{g}$ coming from Lemma \ref{lem.holder}.
\begin{proof}
For the first statement the upper bound follows from \cite[Lemma 2.8]{cantrell.tanaka.1}. For the lower bound, suppose that there is a path from $x_{k-1}$ into the word maximal component $B_i$. Suppose $B_i$ has $l$ states. Choose elements $y^1,\ldots, y^l \in \Sigma_{B_i}$ with each initial state $y^j_0$  being different (i.e. each state in $B_i$ is an initial state for one of the $y^j$ and vice versa).  Then, since $\P_\mathfrak{g}$ is  the function from Lemma \ref{lem.holder} (and so is H\"older continuous), there exists $C >0$ such that
\[
\sum_{x_{k-1}, y_1, \ldots, y_{n-j}}  \exp(-v_\mathfrak{g} d_\mathfrak{g}(o,\ev(x_{m-1}, y_1, \ldots, y_{n-j}))) \ge C \min_{w=1,\ldots, l} \sum_{\sigma^{n-j}(x) = y^w} \exp(-v_\mathfrak{g} \P^{n-j}_\mathfrak{g}(x)).
\]
The required bound then follows from the Lemma \ref{lem.rpf}.
The second statement follows from Lemma \ref{lem.rpf} similarly.
\end{proof}

We now want to understand how $\widehat{\nu}$ assigns mass to different components in the Cannon coding. As an initial observation we note that by Lemma \ref{lem.holder}, Lemma \ref{lem.bounds} and expression (\ref{eq.bound}),  if $x = (x_n)_{n=0}^\infty \in \Sigma_{[\ast]}$ and there is a path from $x_{k-1}$ into a word maximal component then there exists $C >1$ (independent of $x$) such that,
 \begin{equation}\label{eq.cylinder1}
C^{-1}\exp(-v_\mathfrak{g}\P_\mathfrak{g}^m(x)) \le \widehat{\nu}[\ast, x_1, \ldots, x_{k-1}] \le C\exp(-v_\mathfrak{g}\P_\mathfrak{g}^m(x)). 
\end{equation}
We also have the following result.
\begin{lemma}\label{lem.nl}
We have that $\widehat{\nu}$ almost every $x\in\Sigma_{[\ast]}$ enters a word maximal component and never leaves.
\end{lemma}

\begin{proof}
Substituting the bounds from Lemma \ref{lem.bounds} into expression (\ref{eq.bound}) shows that, if there is no path from $x_{k-1}$ into a maximal component, then
\[
\widehat{\nu}[\ast, x_1, \ldots, x_{k-1}] = 0.
\]
%To see this choose $0 < \theta < 1$ larger than the largest value of $ \exp\{\text{P}_\Cc(-v_\mathfrak{g}\P_\mathfrak{g})\} $ over all non-word maximal components $\Cc$.
%Now, if there is no path from $x_{m-1}$ into a word maximal component we have that 
%\[
%\sum_{x_{m-1}, y_1, \ldots, y_{n-j}}  e^{-v_\mathfrak{g} d_\mathfrak{g}(o,\ev(x_{m-1}, y_1, \ldots, y_{n-j}))} \le C \theta^{n-j}
%\]
%for all $n\ge1$ for some $C  >0$ indepedent of $x_{m-1}$. This is because the sum over $x_{m-1}, y_1, \ldots, y_{n-j}$ is restricted to summing over paths belonging to components $\Cc$ on which $ \exp\{\text{P}_\Cc(-v_\mathfrak{g}\P_\mathfrak{g})\} < \theta$.
%Substituting this into equation $(\ref{eq.bound})$ proves the claim.
Summing this equality over (the countably many) paths that enter and eventually leave a word maximal component shows that
\[
\widehat{\nu}( x \in \Sigma_{[\ast]}  : \text{ $x$ enters and leaves a word maximal component} ) = 0.
\]

Further, since the pressure of $-v_\mathfrak{g}\P_\mathfrak{g}$ is strictly less than zero on the non word maximal components, it follows from Lemma \ref{lem.rpf} that there exist $C>0, 0<\theta<1$ such that
\[
\sum_{(\ast,x_1,\ldots,x_n) } \exp(-v_\mathfrak{g}d_\mathfrak{g}(o,\ev(\ast,x_1,\ldots,x_n))) \le C \theta^n \ \ \text{ for all $n\ge 1$}
\]
where the sum is over all paths in $\Sigma_{[\ast]}$ that do not enter a maximal component. It then follows from $(\ref{eq.cylinder1})$ that 
\[
\widehat{\nu}( x \in \Sigma_{[\ast]}  : \text{ $x$ never enters a word maximal component} ) = 0.
\]
Combining this with the fact that $\widehat{\nu}$ gives zero mass to the sequences that enter then leave a maximal component, we deduce that
\[
\widehat{\nu}( x \in \Sigma_{[\ast]}  : \text{ $x$ does not spend infinite time in a word maximal component} ) = 0
\]
and this concludes the proof.
\end{proof}

Following the same calculations and reasoning (considering the pressure of $\P_\mathfrak{g}$) we see that there exists $C >1$ depending only on $x_0$ such that, if $x \in \Sigma$ has $x_0, x_{k-1}$ belonging to a word maximal component  and there exists a path of length $j$ from $\ast$ to $x_0$ then
\begin{equation}\label{eq.cylinder}
C^{-1} \exp(-v_\mathfrak{g} \P_\mathfrak{g}^k(x))  \le \sigma_\ast^j\widehat{\nu}[x_0, \ldots, x_{k-1}] \le C \exp(-v_\mathfrak{g} \P_\mathfrak{g}^k(x)).
\end{equation}
If $x_0$ and $x_{k-1}$ do not both belong to a word maximal component then $\sigma_\ast^j\widehat{\nu}[x_0, \ldots, x_{k-1}] = 0$.
This observation helps us to deduce the following.

\begin{lemma}\label{lem.ac}
Suppose that $x_0$ belongs to a word maximal component $B_i$ and that there exists a path of length $j$ from $\ast$ to $x_0$. Then
$\sigma_\ast^j \widehat{\nu}|_{[x_0]}$ and $\mu^i_{-v_\mathfrak{g} \P_\mathfrak{g}}|_{[x_0]}$ are mutually absolutely continuous where $\mu^i_{-v_\mathfrak{g} \P_\mathfrak{g}}$ is the equilibrium state/Gibbs measure for $-v_\mathfrak{g}\P_\mathfrak{g}$ on the component $\Sigma_{B_i}$.
\end{lemma}

\begin{proof}
Equation $(\ref{eq.cylinder})$ tells us that $\sigma_\ast^j \widehat{\nu}|_{[x_0]}$ satisfies the Gibbs property $(\ref{eq:gibbs})$ on cylinders contained in $[x_0]$. Since $\mu^i_{-v_\mathfrak{g} \P_\mathfrak{g}}|_{[x_0]}$ does too and any measurable subset of $[x_0]$ can be approximated by a disjoint union of cylinders contained in $[x_0]$, the result follows. To be more precise, we need to use
% the fact that $\mu^i_{-v_\mathfrak{g} \P_\mathfrak{g}}|_{[x_0]}$ and $\sigma_\ast^j \widehat{\nu}|_{[x_0]}$ are doubling and to apply 
Vitali's covering lemma, see the proof of  \cite[Theorem 3.6]{THaus} for this argument.
\end{proof}

The important consequence of this result is the following.

\begin{lemma}\label{lem.meascomp}
Suppose that $U \subset \Sigma_{[\ast]}$ has the property that $\widehat{\nu}(U) >0$ then there exists a word maximal component $B_i$ and integer $j \ge 1$ such that $\mu^i_{-v_\mathfrak{g} \P_\mathfrak{g}}(\sigma^j(U)) > 0$.
\end{lemma}
\begin{proof}
For each $k\ge1$ define the set $U_k$ to be the collection of $x \in U$ that enter a word maximal component precisely on their $k$th step and never leaves this component, i.e. $\sigma^k(x) \in \Sigma_{B_i}$ for some $i$ but $\sigma^{j}(x)$ does not belong to any $B_i$ for $j < k$. From Lemma \ref{lem.nl}
\[
\widehat{\nu} \left(U \backslash \bigcup_{k \ge 1} U_k\right) = 0
\]
and so there exists $k$ with $\widehat{\nu}(U_k) >0$. In particular there is a word maximal component $B_i$ and $j \ge 1$ such that
$\widehat{\nu}(\sigma^{-j}(V)) > 0$
where $V =\sigma^{j}(U_j)\cap  \Sigma_{B_i}$.
However by Lemma \ref{lem.ac} there exist $c >0$ and $[x_0] \subset B_i$ with
\[
\mu^i_{-v_\mathfrak{g} \P_\mathfrak{g}}|_{[x_0]}(V) \ge c \cdot \sigma_\ast^j \widehat{\nu}|_{[x_0]}(V) = \widehat{\nu}(\sigma^{-j}(V)\cap [x_0]) >0
\]
as required.
\end{proof}
We immediately deduce the following. 
\begin{corollary}\label{cor.measure}
Suppose that $U \subset \Sigma$ is a set such that $\mu^i_{-v_\mathfrak{g} \P_\mathfrak{g}}(U) =1$ for each word maximal component $B_i$. Then,
\[
\widehat{\nu}\left( \Sigma_{[\ast]} \cap \bigcup_{j=1}^{\infty} \sigma^{-j}(U)\right)  = 1.
\]
\end{corollary}
\begin{proof}
Let
\[
\widetilde{U} = \left(\Sigma_{[\ast]} \cap \bigcup_{j=1}^{\infty} \sigma^{-j}(U)\right)^{c} \ \text{(the power $c$ represents taking the complement).}
\]
Then if $\widehat{\nu}(\widetilde{U} ) >0$ we can apply Lemma \ref{lem.meascomp} to deduce that there is $B_i$ and $j \ge 1$ with $\mu^i_{-v_\mathfrak{g} \P_\mathfrak{g}}(\sigma^j(\widetilde{U} )) > 0$. However this implies that $\mu^i_{-v_\mathfrak{g} \P_\mathfrak{g}}(U^{c}) > 0$ contrary to our assumption.
\end{proof}

\begin{remark}\label{rem.busseman}
We note that the same construction used in this section can be applied to any metric $d \in \Dc_\G$ that has a H\"older continuous Busseman cocycle. We will use this observation when we present generalisations of our main result in Section 
\ref{sec.wm}.
\end{remark}

%%%%%%%%%%%%%%%%%%%%%%%%%%%%%%%%%%%%%%%%%%%%%%%%%%%%%%%%%%%%%%%%%%%%%%%%%%%%%%%%%%%%%%%%%%%%%%%%%%%%%%%%%%%%%%%%%%%%%%%%%%%%%%%%%%%%

\subsection{Concluding the proof}\label{sec.conclude}

We are now ready to conclude the proof of our main result. Let $\widehat{\nu}$ be the measure constructed in the previous section and let the sets $V, U, \widehat{U}$ and  $\widehat{U}_j$ for $j=1,\ldots,k$ be those constructed in Section \ref{sec.set}.

\begin{proof}[Proof of Theorem \ref{thm.ae}]
By Proposition \ref{prop.ssr} there is $D>0$ such that if $\x \in V$ then there exists a $(1,D)$-rough geodesic (for $d_\mathfrak{g}$) representative for $\x$ that satisfies the sparse spectrally rigid property stated in the theorem. It therefore suffices to show that $
\nu_\mathfrak{g}(V) = 1$ to conclude the proof.

Recall that, in the previous section, we used $\mu^j_{-v_\mathfrak{g}\P_\mathfrak{g} }$ to denote the equilibrium state for $-v_\mathfrak{g}\P_\mathfrak{g} $ on the component $B_j$. 
By the Poincar\'e Recurrence Theorem (which we can apply since Gibbs measure are shift invariant) have that $\mu^j_{-v_\mathfrak{g} \P_\mathfrak{g}}(\widehat{U}_j) =1$ for each maximal component $B_j$. Hence we know that $\widehat{\nu}$ almost every $x \in \Sigma_{[\ast]}$ belongs to 
\[
 U = \Sigma_{[\ast]} \cap \bigcup_{l \ge 1}  \sigma^{-l}\left(\widehat{U}\right)
\]
by Corollary \ref{cor.measure}.
The theorem now follows from Lemma \ref{lem.wsc} as this lemma implies that $\nu_\mathfrak{g}(V) = \widehat{\nu}(\ev_\ast^{-1}(V)) \ge \widehat{\nu}(U) = 1$.
\end{proof}

%%%%%%%%%%%%%%%%%%%%%%%%%%%%%%%%%%%%%%%%%%%%%%%%%%%%%%%%%%%%%%%%%%%%%%%%%%%%%%%%%%%%%%%%%%%%%%%%%%%%%%%%%%%%%%%%%%%%%%%%%%%%%%%%%%%%

\section{Generalisations}\label{sec.wm}

In this section we show that Theorem \ref{thm.ae} holds when we swap $d_\mathfrak{g}$ for a word metric.
In fact we can prove a more general rigidity result for a large collection of metrics $\overline{\Dc}_\G$ on 
$\G$. Given a hyperbolic group $\G$, the space $\overline{\Dc}_\G$ consists of all left invariant (pseudo-)metrics $d$ on $\G$ that have non-constant translation length function and for which there exist $\lambda > 0$ and $d_0 \in \Dc_\G$ such that
\[
(x,y)_{o,d} \le \lambda (x,y)_{o,d_0} + \lambda
\]
for all $x,y \in \G$. This space was introduced in \cite{cantrell.reyes.man} and clearly $\overline{\Dc}_\G$ contains $\Dc_\G$. We say that a  subset of conjugacy classes $E \subset \conj(\G)$ is spectrally rigid for $\overline{\Dc}_\G$ if when $d,d_\ast \in \overline{\Dc}_\G$ have the same marked length spectrum on $E$ then they are roughly similar (or equivalently have the same marked length spectrum up to scaling \cite[Theorem 1.2]{cantrell.tanaka.1}). We note that this space includes interesting examples of non-proper metrics on $\G$ and we have the following result.
\begin{theorem}[Theorem 1.6 of \cite{cantrell.reyes.man}] \label{thm.examps}
The following actions give rise to points in $\overline{\Dc}_\G$. 
\begin{enumerate}
    \item Actions on coned-off Cayley graphs for finite, symmetric generating sets, where we cone-off a finite number of quasi-convex subgroups of infinite index.
    \item Non-trivial Bass-Serre tree actions with quasi-convex edge stabilizers of infinite index. More generally, cocompact actions on $\textnormal{CAT}(0)$ cube complexes with quasi-convex hyperplane stabilizers and without global fixed points.
    \item Small actions on $\R$-trees, when $\G$ is a surface group or a free group.
\end{enumerate}
\end{theorem}
When we say `these actions give rise to points in $\overline{\Dc}_\G$' we mean that, when we lift the metric on the space to a metric on $\G$ (as we did for $d_\mathfrak{g}$ in Example \ref{ex.manifold}) this lifted metric belongs to $\overline{\Dc}_\G$.
We can prove the following rigidity result of $\overline{\Dc}_\G$.
\begin{theorem}\label{thm.general}
Let $\G$ be a non-elementary hyperbolic group. Equip $\G$ with a finite generating set $S$ and write $\nu_S$ for the corresponding Patterson-Sullivan measure (for the word metric $d_S$) on the Gromov boundary $\partial \G$ of $\G$. Then there exists $D >0$ such that for $\nu_S$ almost every $\xi \in \partial \G$ there exists a $D$-rough geodesic ray $(\xi_k)_{k=0}^\infty$ for $d_S$ starting at the identity in $\G$  with end point  $\xi$ such that
for any $M\ge 1$ the set $
\{[\xi_k]: k\ge M\}$ is spectrally rigid for $\overline{\Dc}_\G$.

Furthermore, for any function $f: \R_{>0} \to \R_{>0}$ with $f(T) \to \infty$ as $T\to\infty$ there exists a subsequence $n_k$ such that
$E= \{ [\xi_{n_k}] : k\ge 1\}$
is spectrally rigid for $\overline{\Dc}_\G$ and
\[
\#\{ \ell_S[x] < T : [x] \in E\} \le f(T)
\]
for all $T>0$.
\end{theorem}
\begin{remark}
This result also holds if we replace the word metric $d_S$ with a Green metric (or in fact any \textit{strongly hyperbolic metric}, see Definition 2.1 in \cite{cantrell.new}) associated to an admissible, finitely supported, symmetric random walk. In this case the Patterson-Sullivan measure is the same as the hitting measure for the random walk on the boundary of the group. The same proof as that presented in the previous section can be applied since the construction in Section \ref{sec.ps} can be carried out as discussed in Remark \ref{rem.busseman}.
%The spectral rigid sets in this theorem can also be used to distinguish some non-proper actions of hyperbolic groups on metric spaces. Our spectrally rigid sets are in fact rigid for $\overline{\Dc}_\G$, a space of metric structures introduced in \cite{cantrell.reyes.man}. This space includes certain actions on $\text{CAT}(0)$ cube complexes as well as conned off Cayley graphs. We refer the reader to \cite{cantrell.reyes.man} for more details. 
\end{remark}
The proof follows exactly the same lines used to prove Theorem \ref{thm.ae} and so we outline the changes needed instead of rewriting the whole proof. In fact, the set $V$ from Definition \ref{def.v} is precisely the same spectrally rigid set appearing in Theorem \ref{thm.general}.  There are two main differences in the argument  to show that $V$ satisfies the conditions in Theorem \ref{thm.general}.\\

\textbf{First change:} We need to prove an analogue of Proposition \ref{prop.ssr} for $d_S$ and for metrics in $\overline{\Dc}_\G$. To do the we use the fact that for each $d \in \overline{\Dc}_\G \backslash \Dc_\G$ there exist $d_1,d_2 \in \Dc_\G$ such that (up to a uniformly bounded error) $d$ can be written as 
\[
d(x,y) = \textnormal{Dil}(d_1,d_2)d_2(x,y) - d_1(x,y) \ \text{ for all $x,y \in \G$ where $\textnormal{Dil}(d_1,d_2) = \sup_{[x]\in\conj}\frac{\ell_{d_1}[x]}{\ell_{d_{d_2}}[x]}$}
\]
and the supremum is over all non-torsion conjugacy classes.
This is shown in \cite[Proposition 5.1]{cantrell.reyes.man}. In particular the translation length function for $d$ is a linear combination of those for $d_1,d_2$:
\[
\ell_d[x] = \textnormal{Dil}(d_1,d_2) \ell_{d_2}[x] - \ell_{d_1}[x]
\]
for all $x \in \G$. The key point here is that the Gromov product for $d \in \overline{\Dc}_\G \backslash \Dc_\G$ can be controlled by the Gromov product for the corresponding metrics $d_1,d_2 \in \Dc_\G$ (and hence by the Gromov product for $d_S$). Using this fact, the proof of Proposition \ref{prop.ssr} can be recreated for metrics in $\overline{\Dc}_\G$. We leave the details to the reader.\\

 \textbf{Second change:} We then need to show that $\nu_S(V) =1$. To do this we need to know that there  is a sufficiently nice measure $\widehat{\nu}_S$ on $\Sigma_{[\ast]}$ that pushes forward under $\ev_\ast$ to a Patterson-Sullivan measure (as in Proposition \ref{prop.qcm}) $\nu_S$ for $d_S$ on $\partial \G$, i.e. so that $\nu_S(U) = \widehat{\nu}_S(\ev_\ast^{-1}(U))$ for $U \subset \partial \G$. The measure $\widehat{\nu}$ can be constructed as in Section $4$ of \cite{CalegariFujiwara2010} (see also \cite{georays}). We will not provide the construction here but will instead present the properties of $\widehat{\nu}_S$ that we require for our proofs. Fix a Cannon coding $\Sigma$ for $\G$ and a finite generating set $S$. 

\begin{proposition}\label{prop.mme}
There exists a measure $\mu$ on $\Sigma$ such that
\[
 \lim_{n\to\infty}\frac{1}{n} \sum_{k=0}^n \sigma_\ast^k\widehat{\nu}_S =  \mu.
\]
Further there exist $0< \alpha_i <1$ for $i=1,...,m$ with $\sum_{i=1}^m \alpha_i =1$ such that
\begin{equation*}
\mu = \sum_{i=1}^m \alpha_i \mu_i,
\end{equation*}
where each $\mu_i$ is the measure of maximal entropy for the system $(\Sigma_{B_i},\sigma)$.
\end{proposition}
Here the measure of maximal entropy on each component is the Gibbs state for the constant potential $\P =1$ that appears in Proposition \ref{prop.vp}.
An important consequence of this is that if $U \subset \Sigma$ has $\mu_i(U\cap \Sigma_{B_i}) = 1$ for each $i=1,\ldots, k$ then the set
\[
\Sigma_{[\ast]} \cap \bigcup_{j=1}^\infty \sigma^{-j}(U)
\]
has full $\widehat{\nu}_S$ measure \cite[Proposition 4.6]{georays}. This is an analogue of Corollary \ref{cor.measure} for $\widehat{\nu}_S$ and following the same proof as that presented in Section \ref{sec.conclude}, we see that $\nu_S(V) = 1$. Combining this with the first change above concludes the proof for Theorem \ref{thm.general}.\\

%Using these two new inputs we can deduce Theorem \ref{thm.general} by following the proof of Theorem \ref{thm.ae}. 
%In fact the set $V$ from Definition \ref{def.v} is precisely the same spectrally rigid set appearing in Theorem \ref{thm.general}. As the argument is almost identical, we leave the details to the reader.\\

As an immediate corollary of Theorem \ref{thm.general} we deduce a refinement of Kapovich's result \cite{kap} for Outer Space. We refer the reader to \cite{CV} for an introduction to Outer Space. Let $F_N$ be the free group on $N \ge 2$ generators. We let $\text{cv}_N$ denote Culler-Vogtmann Outer Space and we recall that points in Outer Space are determined by their length spectrum, i.e. if $T_1, T_2 \in \text{cv}_N$ have the same length spectra then $T_1 = T_2$. We say that a subset of conjugacy classes $E \subset \conj(F_N)$ is spectrally rigid for Outer Space if when $T_1, T_2 \in \text{cv}_N$ have the same length spectra on $E$ then $T_1 = T_2$.
\begin{theorem}\label{thm.os}
Let $F_N$ be the free group on $N \ge 2$ generators. Equip $F_N$ with a finite generating set $S$ and write $\nu_S$ for the corresponding Patterson-Sullivan measure (for the word metric $d_S$) on the Gromov boundary $\partial F_N$ of $F_N$. Then there exist $D > 0$ such that for $\nu_S$ almost every $\xi \in \partial \G$ there exists a $D$-rough geodesic ray $(\xi_k)_{k=0}^\infty$ for $d_S$ starting at the identity in $F_N$ with end point $\xi$ such that
for any $M\ge 1$ the set $
\{[\xi_k]: k\ge M\}$ is spectrally rigid for Outer Space.

Furthermore, for any function $f: \R_{>0} \to \R_{>0}$ with $f(T) \to \infty$ as $T\to\infty$ there exists a subsequence $n_k$ such that
$E= \{ [\xi_{n_k}] : k\ge 1\}$
is spectrally rigid for Outer Space and
\[
\#\{ [x] \in E : \ell_S[x] < T \} \le f(T)
\]
for all $T>0$.
\end{theorem}

\begin{remark}
In fact, by Theorem \ref{thm.examps}, this result generalises  to small actions of free groups and surface groups on $\mathbb{R}$-trees. That is, for any fixed word metric $d_S$ on $F_N$ (or a surface group) there are arbitrarily small length spectrum rigidity sets for isometric, small actions of $F_N$ (or the surface group) on $\mathbb{R}$-trees (that exist in $\nu_S$ almost every direction in $F_N$). This provides an interesting example of sparse length spectrum rigidity sets for non-proper actions. 
\end{remark}

We can also apply our methods to Hitchin representations. For an introduction to Hitchin representations we refer the reader to \cite{HitchNotes}. Suppose that $\G$ is a surface group (the fundamental group of a closed hyperbolic surface) and that $\rho: \G \to \PSL_{n_1}(\R)$ is a  Hitchin representations. For $A \in \PSL_k(\R)$  let $\lambda_1(A)$ denote the spectral radius of $A$. In this setting the map $[x] \mapsto \log\lambda_1(\rho(x))$ from $\conj(\G)$ to $\R$ plays the role of the marked length spectrum for the $\rho$. We say that a subset $E \subset \conj(\G)$ is a spectrally rigid for Hitchin representations if when two Hitchin representations $\rho_1, \rho_2$ have the same length spectra on $E$ then they have the same length spectra everywhere (or equivalently either $\rho_1 = \rho_2$ or $\rho_1$ is the contragredient of $\rho_2$ \cite{BCLS}). We then have the following.

\begin{theorem}\label{thm.hitchin}
Let $\G$ be a surface group. Equip $\G$ with a finite generating set $S$ and write $\nu_S$ for the corresponding Patterson-Sullivan measure (for the word metric $d_S$)  on the Gromov boundary $\partial \G$ of $\G$. Then there exists $D >0$ such that for $\nu_S$ almost every $\xi \in \partial \G$ there exists a $D$-rough geodesic ray $(\xi_k)_{k=0}^\infty$ for $d_S$ starting at the identity in $\G$ with end point $\xi$ 
such that
for any $M\ge 1$ the set $
\{[\xi_k]: k\ge M\}$ is spectrally rigid for Hitchin representations.

Furthermore, for any function $f: \R_{>0} \to \R_{>0}$ with $f(T) \to \infty$ as $T\to\infty$ there exists a subsequence $n_k$ such that
$E= \{ [\xi_{n_k}] : k\ge 1\}$
is spectrally rigid for Hitchin representations and
\[
\#\{ [x] \in E : \ell_S[x] < T \} \le f(T) 
\]
for all $T>0$.
\end{theorem}
The proof again follows exactly the same argument used to prove Theorem \ref{thm.ae} and so we omit it.

\subsection*{Open access statement}
For the purpose of open access, the author has applied a Creative Commons Attribution (CC BY) licence to any Author Accepted Manuscript version arising from this submission.

\bibliographystyle{alpha}
\bibliography{SparceLSRPS}

\end{document}